\documentclass[12pt, 
]{amsart}

\usepackage[all]{xy}
\usepackage{mathrsfs, amsthm, amscd, amsmath, amssymb, graphicx}
\usepackage[dvipdfmx]{hyperref}

\newtheorem{theo}{Theorem}[section]

\newtheorem{claim}[theo]{Claim}

\newtheorem{conj}[theo]{Conjecture}
\theoremstyle{definition}
\newtheorem{defi}[theo]{Definition}

\theoremstyle{remark}
\newtheorem{rem}[theo]{Remark}

\makeatletter
    
    \@addtoreset{equation}{section}
\makeatother

\newcommand{\dbar}{\overline{\partial}}
\newcommand{\e}{\varepsilon}

\newcommand{\I}[1]{\mathcal{I}(#1)}
\newcommand{\nd}[1]{{\rm{nd}}(#1)}
\newcommand{\ndrel}[1]{{\rm{nd_{rel}}}(#1)}

\usepackage{geometry}
\usepackage{setspace} 
\begin{document}

\title[Variation of numerical dimension of singular hermitian line bundles]
{Variation of numerical dimension\\ of singular hermitian line bundles}

\author{Shin-ichi MATSUMURA}

\address{Mathematical Institute, Tohoku University, 
6-3, Aramaki Aza-Aoba, Aoba-ku, Sendai 980-8578, Japan.}

\email{{\tt mshinichi-math@tohoku.ac.jp, mshinichi0@gmail.com}}

\date{\today, version 0.01}

\renewcommand{\subjclassname}{%

\textup{2010} Mathematics Subject Classification}



\dedicatory{Dedicated to Professor Kang-Tae Kim 
on the occasion of his 60th birthday}

\maketitle

\begin{abstract}
The purpose of this paper is to give two supplements for vanishing theorems\,$:$ 
One is a relative version of the Kawamata-Viehweg-Nadel type vanishing theorem, 
which is obtained from an observation for the variation of the numerical dimension of singular hermitian line bundles. 
The other is an analytic injectivity theorem for log canonical pairs on surfaces, 
which can be seen as a partial answer for Fujino's conjecture.  
\end{abstract}


\section{Introduction}\label{Sec1}
The injectivity theorem 
is one of the most important generalizations of the Kodaira vanishing theorem, 
and it plays an important role in complex geometry in the last decades, 
in which analytic methods and algebraic geometric methods 
have been nourishing each other. 
After Tankeev's pioneer work in \cite{Tan71}, 
Koll\'ar in \cite{Kol86a} and \cite{Kol86b} 
established the celebrated injectivity theorem 
for semi-ample line bundles on projective manifolds by Hodge theory. 
Enoki in \cite{Eno90} generalized Koll\'ar's injectivity for semi-positive line bundles on compact K\"ahler manifolds by the theory of harmonic integrals, 
and Takegoshi in \cite{Tak95} gave a relative version of Enoki's injectivity 
for K\"ahler morphisms. 
We recently obtained a further generalization of them for 
pseudo-effective line bundles with singular hermitian metrics 
by a combination of the theory of harmonic integrals and $L^2$-methods for $\dbar$-equations (for example see \cite{FM16}, \cite{Mat15}, \cite{Mat16}, and \cite{Mat13}).

In this paper, 
as an application of \cite{Mat16}, 
we prove a relative version of the Kawamata-Viehweg-Nadel type 
vanishing theorem (Theorem \ref{KVN}) for K\"ahler morphisms, 
by the vanishing theorem on compact K\"ahler manifolds in \cite{Cao14}, 
the solution of the strong openness conjecture in \cite{GZ15a} 
(see also \cite{GZ15b}, \cite{Hie14}, \cite{Lem14}), 
and an observation for the variation of the numerical dimension of 
singular hermitian line bundles (Theorem \ref{num}). 
Moreover, as an application of \cite{Mat17}, 
we give an affirmative answer for Fujino's conjecture (Conjecture \ref{F-conj}) 
in the two dimensional cases (Theorem \ref{mainn}).

\begin{theo}[{Variation of numerical dimensions, cf. \cite[Proposition 1.6]{Mat16}}]\label{num} 
Let $\pi \colon X \to B$ be a smooth proper K\"ahler morphism 
from a  complex manifold $X$ to a complex manifold $B$, 
and let $T$ be a positive $d$-closed $(1,1)$-current on $X$. 
Then there is a subset $C \subset B$ of Lebesgue measure zero with the following property\,$:$
\vspace{0.1cm}
\\
For an arbitrary point $b \in B \setminus C$, the restriction $T|_{X_{b}}$ of $T$ 
to the fiber $X_{b}:=\pi^{-1}(b)$ is well-defined, and  
the numerical dimension $\nd{X_b, T|_{X_{b}}}$ of $T|_{X_{b}}$ on $X_b$
does not depend on $b \in B \setminus C$. 
\end{theo}

\begin{theo}[Relative vanishing theorem of Kawamata-Viehweg-Nadel type] \label{KVN} 
Let $\pi \colon X \to \Delta$ be a surjective proper K\"ahler morphism 
from a complex manifold $X$ to an analytic space $\Delta$, 
and $(F,h)$ be a singular hermitian line bundle on $X$ with semi-positive curvature.

Then we have 
$$
R^q \pi_{*}(K_{X}\otimes F \otimes \I{h})=0
\quad \text{for every $q>n -\ndrel{F,h} $}, 
$$
where $n$ is the relative dimension of $\pi \colon X \to \Delta$ and 
$\ndrel{F,h}=\ndrel{\sqrt{-1}\Theta_{h}(F)}$ is the relative numerical dimension defined 
by Definition \ref{def-rel}.  
Here $K_X$ is the canonical bundle on $X$, 
$\I{h}$ is the multiplier ideal sheaf of $h$, 
$R^q \pi_{*}(\bullet)$ is the $q$-th direct image sheaf. 
\end{theo}

Ambro and Fujino proved an injectivity theorem 
for log canonical (lc for short) pairs by Hodge theory 
(see \cite{Amb03}, \cite{Amb14}, \cite{EV92}, 
\cite[Section 6]{Fuj11}, \cite{Fuj12b}, and \cite{Fuj13b}). 
It is a natural and quite interesting problem 
to ask whether the injectivity theorem for lc pairs 
can be generalized from semi-ample line bundles to semi-positive line bundles, 
which was first posed by Fujino.

\begin{conj}[{\cite[Conjecture 2.21]{Fuj15b}, cf. \cite[Problem 1.8]{Fuj13a}}] 
\label{F-conj}
Let $D$ be a simple normal crossing divisor on a compact K\"ahler manifold $X$ 
and $F$ be a semi-positive line bundle on $X$ 
$($that is, it admits a smooth hermitian metric with semi-positive curvature$)$. 
Assume that there is a section $s \in H^{0}(X, F^m)$ such that 
the zero locus $s^{-1}(0)$ 
contains no lc centers of the lc pair $(X,D)$. 
Then, the multiplication map induced by the tensor product with $s$ 
\begin{equation*}
H^{q}(X, K_{X} \otimes D \otimes F) 
\xrightarrow{\otimes s} 
H^{q}(X, K_{X} \otimes D \otimes F^{m+1} )
\end{equation*}
is injective for every $q$. 
\end{conj}

In \cite{Mat17}, we proved the above conjecture 
in the case of purely log terminal pairs 
by developing techniques in \cite{Mat13} and \cite{Tak97} 
(see \cite{Che15}, \cite{HLWY16}, \cite{KSX17} for another interesting approach). 
In this paper, as an application of \cite{Mat17}, 
we affirmatively solve Fujino's conjecture on surfaces 
without any assumptions (Theorem \ref{mainn}). 

\begin{theo}\label{mainn}
Let $D$ be a simple normal crossing divisor on a compact K\"ahler surface $X$ 
and $F$ be a semi-positive line bundle on $X$. 
Assume that there is a section $s \in H^{0}(X, F^m)$ such that 
the zero locus $s^{-1}(0)$ contains no irreducible components of $D$. 
Then the same conclusion as in Conjecture \ref{F-conj} holds. 
\end{theo}

\begin{rem}\label{rem-mainn}
The assumption of Theorem \ref{mainn} for the zero locus $s^{-1}(0)$ is weaker than 
that of Conjecture \ref{F-conj}, 
since the zero locus $s^{-1}(0)$ may contain $D_{i} \cap D_j$, 
where $D_i$ is an irreducible component of $D=\sum_{i \in I}D_{i}$. 
\end{rem}

\subsection*{Acknowledgements}

This paper has been written during author's stay in 
Institut de Math\'ematiques de Jussieu-Paris Rive gauche (IMJ-PRG). 
The author would like to thank the members of IMJ-PRG for their hospitality.
He is supported by the Grant-in-Aid 
for Young Scientists (A) $\sharp$17H04821 from JSPS 
and the JSPS Program for Advancing Strategic International Networks 
to Accelerate the Circulation of Talented Researchers.

\section{Proof of the results}\label{Sec2}

\subsection{Proof of Theorem \ref{num}}\label{Sec2-1}

We first recall the definition of the numerical dimension  
of (possibly) singular hermitian line bundles with semi-positive curvature 
(more generally positive $d$-closed $(1,1)$-currents) in \cite{Cao14}.  
For a positive $d$-closed $(1,1)$-current $T$ on a compact K\"ahler manifold $X$, 
we can take a family $\{T_{k}\}_{k=1}^{\infty}$ of 
$d$-closed $(1,1)$-currents (called an equisingular approximation with analytic singularities) with the following properties\,$:$
\begin{itemize}
\item[(0)] $T_{k}$ is a $d$-closed $(1,1)$-current representing the same cohomology class $\{T\}$.
\item[(1)] $T_{k}$ has analytic singularities. 
\item[(2)] $T_k \geq -\e_k \omega$, where $\e_{k} \searrow 0$ and $\omega$ is a fixed K\"ahler form on $X$.
\item[(3)] $T_{k+1}$ is more singular than $T_k$. 
\item[(4)] For rational numbers $\delta>0$ and $m>0$, there exists an integer $k_{0}$ 
such that 
$$\I{(m (1+\delta) T_{k})} \subset \I{m T} \text{ for } k \geq k_{0}.$$ 
\end{itemize}
Here, for a $d$-closed $(1,1)$-current $S$, 
the multiplier ideal $\I{S}$ can be defined by 
the set of holomorphic functions $f$ such that $|f|^{2}e^{-\varphi}$ is locally integrable, 
where $\varphi$ is  a local potential function of $S$. 
Then, by the same way as in \cite{Cao14}, the numerical dimension $\nd{T}$ of $T$ is defined by 
$$
\nd{T} :=\nd{X,T} := \max\{ \nu \in \mathbb{Z}_{\geq 0}\, |\,  \liminf_{k \to \infty} 
\int_{X} (T_{k, \rm{ac}})^{\nu} \wedge \omega^{n-\nu}>0\}, 
$$
where $n$ is the dimension of $X$ and $T_{\rm{ac}}$ is 
the absolutely continuous part of $T$ (see \cite{Bou02} for the definition). 
Note that the above numerical dimension can be expressed 
by the growth of the dimension of the space of global sections,  
in the case where $X$ is a projective manifold and $T$ is a semi-positive curvature current 
of a singular hermitian line bundle (see \cite[Proposition 4.3]{Cao14}).

\begin{proof}[Proof of Theorem \ref{num}]
By replacing $B$ with an open subset in $B$ (if necessarily), 
we may assume that $X$ is a K\"ahler manifold. 
Further we may assume that 
there is an equisingular approximation $\{T_{k}\}_{k=1}^{\infty}$ of $T$ 
satisfying properties (0)--(4) on a (non-compact) manifold $X$, 
since Demailly's approximation theorem (see \cite{Dem92}, \cite{DPS01}), 
which plays a crucial role to obtain such an equisingular approximation, 
still works on a relatively compact set in $X$. 
By an observation for multiplier ideal sheaves and Fubini's theorem,  
we have the following claim\,$:$

\begin{claim}\label{claim}
There is a subset $C \subset B$ of Lebesgue measure zero with the following property\,$:$
For an arbitrary point $b \in B \setminus C$, the restriction $T|_{X_{b}}$ 
$($resp. $T_k|_{X_{b}}$$)$ of $T$ $($resp. $T_k$$)$ to the fiber $X_{b}:=\pi^{-1}(b)$ is well-defined 
$($that is, the restriction of its potential function is not identically $-\infty$$)$, 
and $\{T_k|_{X_{b}}\}_{k=1}^{\infty}$ 
is also an equisingular approximation of $T|_{X_{b}}$ on $X_b$. 
\end{claim}

\begin{proof}[Proof of Claim \ref{claim}]
We can easily check that 
properties (0)--(3) for $T_{k}|_{X_{b}}$ and $T|_{X_{b}}$ 
still hold on a fiber $X_{b}$, if the restriction of them to $X_b$ is well-defined. 
In general, for a quasi-psh function $\varphi$ on $X$, 
we have the restriction formula $\I{\varphi |_{X_b}} \subset \I{\varphi }|_{X_b}$
by the Ohsawa-Takegoshi $L^2$-extension theorem. 
Further, by Fubini's theorem, we can show that the converse inclusion holds for almost all $b \in B$, 
that is, the subset 
$$
\{b\in B \,|\, \varphi|_{X_{b}} \text{ is not well-defined or } 
\I{\varphi|_{X_{b}}} \not = \I {\varphi}|_{X_{b}} \}
$$
has Lebesgue measure zero. 
Indeed, for a holomorphic function $f$ 
on a (sufficiently small) open set $U$ in $X$, 
Fubini's theorem yields  
$$
\int_{U} |f|^2(z,b){e^{-\varphi(z,b)}} 
= \int_{b\in B} \Big( \int_{z \in X_b \cap U} |f|^2(z,b){e^{-\varphi(z,b)}} \Big),  
$$
where $(z, b)$ is a coordinate on $U$ 
such that $b=\pi(z,b)$ gives a local coordinate on $B$. 
If the left hand side converges, the integrand 
$\int_{z \in X_b \cap U} |f|^2(z,b){e^{-\varphi(z,b)}}$ also converges for almost all $b \in B$. 
This implies that the above subset has Lebesgue measure zero
since multiplier ideal sheaves are coherent sheaves. 

Now we define the subset $C$ by the union of 
\begin{align*}
&\{b\in B \,|\, T|_{X_{b}} \text{ is not well-defined or } 
\I{mT|_{X_{b}}} \not = \I {mT}|_{X_{b}} \} \text{ and }\\
&\{b\in B \,|\, T_k|_{X_{b}} \text{ is not well-defined or } 
\I{m(1+\delta)T_k|_{X_{b}}} \not = \I {m(1+\delta)T_k}|_{X_{b}} \}, 
\end{align*}
where $\delta$ and $m$ run through positive rational numbers. 
Then the union $C$  also has Lebesgue measure zero 
since $C$ is a countable union of subsets of Lebesgue measure zero. 
Therefore it follows that 
the restriction $T_{k}|_{X_{b}}$ gives an equisingular approximation 
of $T|_{X_{b}}$ on the fiber $X_{b}$ for an arbitrary point $b \in B \setminus C$.  
\end{proof}

Since $T_k$ has analytic singularities, 
we can take a modification $f_{k}\, \colon \, X_{k} \to X$ such that 
$f_{k}^{*}(T_{k}) = P_{k} + [E_{k}]$, 
where $P_{k}$ is a smooth semi-positive $(1,1)$-form on $X_k$ and $[E_{k}]$ is the integration current 
of an effective $\mathbb{R}$-divisor $E_{k}$. 
We consider the restriction of $f_{k}$ to $X_{k,b}:=f_{k}^{-1}(X_b)$
$$
f_{k,b}:= f_{k}|_{f_{k}^{-1}(X_b)}\, \colon \, X_{k, b}:= f_{k}^{-1}(X_b) \to X_b. 
$$
Let $Z_{k}$ be a subvariety of $X$ such that 
$f_{k} \, \colon \, X_{k} \setminus f_{k}^{-1}(Z_k) \to X \setminus Z_{k}$ 
is an isomorphism. 
Since the subset $C_k:=\{b \in B \,|\, X_{b} \subset Z_{k} \text{ or } X_{k, b} \subset E_k \}$ 
is a proper subvariety of $B$, 
by replacing $C$ with $\cup_{k=1}^{\infty} C_{k} \cup C$ (if necessarily), 
we may assume that 
the restriction $f_{k,b} \, \colon \, X_{k, b}\to X_b$ is 
a modification and the fiber $X_{k, b}$ is not contained in $E_{k}$ for every $b \in B \setminus C$. 
Then, for every $b \in B \setminus C$, we have 
$$
f_{k,b}^{*}(T_{k}|_{X_{k,b}}) = f_{k}^{*}(T_{k})|_{X_{k,b}} = P_{k}|_{X_{k,b}} + [E_{k} |_{X_{k,b}}]. 
$$
Then, for a K\"ahler form $\omega$ on $X$ and a non-negative integer $d$, 
we can see that 
\begin{align*}
\int_{X_b} (T_{k}|_{X_{b}})_{\rm{ac}}^{d} \wedge \omega|_{X_{b}}^{n-d} 
&= \int_{X_b \setminus {\rm{Sing}}\, T_{k}|_{X_{b}}} (T_{k}|_{X_{b}})^d \wedge \omega|_{X_{b}}^{n-d} \\
&= \int_{X_{k,b} \setminus {\rm{Supp}} E_{k}} (P_{k}|_{X_{k, b}})^d \wedge 
f_{k,b}^{*} \omega|_{X_{b}}^{n-d} \\
&= \int_{X_{k,b}} (P_{k}^d  \wedge f_{k}^{*} \omega^{n-d})|_{X_{k, b}}, 
\end{align*}
where $n$ is the dimension of $X_b$.

Now we consider 
the push-forward 
${(\pi \circ f_{k})}_* (P_{k}^{d} \wedge f_{k}^{*} \omega^{n-d})$ of 
the smooth $(n,n)$-form. 
This push-forward is a $d$-closed $(0,0)$-current, 
and thus it must be a constant function on $B$. 
Let $C'_{k}$ be a subvariety of $B$ such that 
$\pi \circ f_{k} $ is a smooth morphism over $B \setminus C'_{k}$. 
Then the push-forward 
${(\pi \circ f_{k})}_* (P_{k}^{d} \wedge f_{k}^{*} \omega^{N-d})$ is 
a smooth function whose value at $b \in B \setminus C'_{k}$ is given by the fiber integral. 
Therefore, replacing $C$ with $\cup_{k=1}^{\infty} C'_{k} \cup C$ again, 
we can check that  
$$
{(\pi \circ f_{k})}_* (P_{k}^{d} \wedge f_{k}^{*} \omega^{n-d})(b) 
=\int_{X_{k,b}} (P_{k}^d  \wedge f_{k}^{*} \omega^{n-d})|_{X_{k, b}}
=\int_{X_b} (T_{k}|_{X_{b}})_{\rm{ac}}^{d} \wedge \omega|_{X_{b}}^{n-d} 
$$
for $b \in B \setminus C$ by the above argument. 
The left hand side does not depend on $b \in B \setminus C$ 
since ${(\pi \circ f_{k})}_* (P_{k}^{d} \wedge f_{k}^{*} \omega^{n-d})$ is 
a constant function. 
Hence we obtain the desired conclusion. 
\end{proof}

\subsection{Proof of Theorem \ref{KVN}}
As an application of 
\cite[Theorem 1.3]{Cao14}, \cite[Theorem 1.1]{GZ15a}, and \cite[Theorem 1.1]{Mat16}, 
we prove Theorem \ref{KVN} by the same argument as in \cite[Theorem 1.7]{Mat16}. 
We first define the relative numerical dimension by using Theorem \ref{num}. 

\begin{defi}[Relative numerical dimension]
\label{def-rel}
Let $\pi \colon X \to \Delta$ be a surjective proper K\"ahler morphism 
from a complex manifold $X$ to an analytic space $\Delta$, and 
let $T$ be a positive $d$-closed $(1,1)$-current on $X$. 
For a Zariski open set $B$ over which $\pi$ is smooth, 
by taking $C \subset B$ of Lebesgue measure zero satisfying the property of 
Theorem \ref{num}, 
we define the relative numerical dimension $\ndrel{T} $ 
$$
\ndrel{T} := \nd{X_b, T|_{X_{b}}}
$$
for $b \in B \setminus C$. 
(See subsection \ref{Sec2-1} for the definition of the usual numerical dimension.) 
\end{defi}

\begin{proof}[Proof of Theorem \ref{KVN}]
For a Zariski open set $B$ in $C$ over which $\pi$ is smooth, 
we take $C \subset B$ of Lebesgue measure zero with the property of Theorem \ref{num}. 
We replace $C$ with $C \cup (\Delta \setminus B)$. 
Then, by the argument of Claim \ref{claim}, 
we obtain the additional property\,$:$ 
$$\text{
$\nd{F|_{X_{b}},h|_{X_{b}}} = \ndrel{F,h}$ 
and $\I{h|_{X_{b}} } = \I{h} |_{X_{b}}$ holds 
for every $b \in \Delta \setminus C$. }
$$
For $q>n-\ndrel{F,h}$ and for $b \in \Delta \setminus C$, 
we have the vanishing theorem 
\begin{align*}
H^{q}(X_{b}, \mathcal{O}_{X_{b}}(K_{X} \otimes F)\otimes \I{h})&= 
H^{q}(X_{b}, \mathcal{O}_{X_{b}}(K_{X} \otimes F)\otimes \I{h|_{X_{b}}})
=0
\end{align*}
on a fiber $X_b$ by \cite[Theorem 1.3]{Cao14} and \cite[Theorem 1.1]{GZ15a}. 
In particular, for the Zariski open set $\Delta'$ in $\Delta$ defined by 
$$
\Delta':=\{b \in \Delta \, |\, \text{$\pi$ is smooth at $b$ and 
$R^{q}\pi_{*}(K_{X}\otimes F \otimes \I{h})$ is locally free at $b$}\}, 
$$ 
we can see that $R^{q} \pi_{*}(K_{X} \otimes F \otimes \I{h})_{b}=0$ 
for every $b \in \Delta' \setminus C$ 
by the flat base change theorem. 
Hence we have  
$R^{q}\pi_{*}(K_{X}\otimes F \otimes \I{h})=0$ on $\Delta'$. 
We obtain the desired conclusion 
since $R^{q}\pi_{*}(K_{X}\otimes F \otimes \I{h})$ 
is torsion free by \cite[Theorem 1.1]{Mat16}. 
\end{proof}

\subsection{Proof of Theorem \ref{mainn}}
We finally prove Theorem \ref{mainn} as an application of \cite{Mat17}.

\begin{proof}[Proof of Theorem \ref{mainn}]
The conclusion is obvious in the case  $q=0$. 
Further we can easily check the conclusion in the case  $q=2$ 
by the Serre duality. 
Indeed, by the Serre duality, we have  
$$
H^{2}(X, K_X \otimes D \otimes F)=H^{0}(X,  \mathcal{O}_{X}(D \otimes F)^{*})=0
$$ 
unless $s$ is a non-vanishing section. 
Hence it is enough to consider the case  $q=1$. 

Let $\alpha$ be a cohomology class $\alpha \in H^{q}(X, K_X \otimes D \otimes F)$ 
satisfying $s\alpha=0  \in H^{q}(X, K_X \otimes D \otimes F^{m+1})$. 
By \cite[Theorem 1.6]{Mat17}, it is sufficient to show that 
$\alpha$ belongs to the image of 
the morphism 
$$
\theta_{D}: H^{q}(X, K_X \otimes F) \to H^{q}(X, K_X\otimes D \otimes F)
$$
induced by the effective divisor $D$. 

For the irreducible decomposition $D=\sum_{i \in I}D_{i}$ of $D$, 
we define the divisors $D_J$ and $D_K$ by 
$$
D_{J}:=\sum_{j\in J} D_{j}   \text{ and } D_{K}:=D -D_{J}, 
\text{ where } J:=\{ i \in I \,|\, D_{i} \cap s^{-1}(0) \ \not= \emptyset \}. 
$$
We consider the commutative long exact sequence induced 
by the standard short exact sequence\,$:$ 
\begin{align}
\vcenter{ 
\xymatrix{
&\ar[d] & \ar[d]\\
&H^q(X,K_{X}\otimes D_J \otimes F )
\ar[d]^-{\theta_K} \ar[r]^-{\otimes s} 
&H^q(X, K_{X} \otimes D_J \otimes F ^{m+1})\ar[d]\\ 
&H^q(X, K_{X}\otimes D \otimes F)
\ar[d]^-{r_K}\ar[r] ^-{\otimes s}
&H^q(X, K_{X}\otimes D \otimes F ^{m+1}) \ar[d]\\ 
&H^q(D_K, K_{D_K} \otimes D_J \otimes F)
\ar[d]\ar[r]^-{f:=\otimes s|_{D_{K}}} 
& H^q(D_K, K_{D_K}\otimes  D_J \otimes F ^{m+1})\ar[d]\\ 
& & 
}}
\end{align}
Here $\theta_{K}$ (resp. $r_{K}$) is the morphism induced by the effective divisor $D_K$ 
(resp. the restriction to $D_{K}$). 
It follows that that $f(r_K(\alpha))=0$ from the assumption $s\alpha=0$. 
On the other hand, the morphism $f$ admits the inverse map 
since the section $s$ is non-vanishing on $D_{K}$ by the definition of $D_K$.  
Therefore we can find $\beta \in H^q(X,K_{X}\otimes D_J \otimes F )$ such that 
$\alpha = \theta_K(\beta)$. 

For a given index $i \in J$, 
we consider $r_{i} (\beta) \in H^q(D_i,K_{D_{i}} \otimes \hat{D}_j \otimes F )$, 
where $\hat{D}_i:=D_{J} - D_i$ and $r_i$ is the morphism induced by the restriction to $D_{i}$. 
It follows that $\deg F|_{D_i} > 0$ since $s^{-1}(0)$ intersects with $D_i$ by the definition of $J$, 
and thus we have $H^q(D_i,K_{D_{i}} \otimes \hat{D}_j \otimes F )=0$ 
by the vanishing theorem on the curve $D_i$. 
In particular, we can take $\beta' \in H^q(X, K_X \otimes \hat{D}_i \otimes F )$ 
such that $\beta = \theta_{i}(\beta')$. 
For $j \in J$ with $j\not= i$, 
it can be seen that  
$$
r_{j}(\beta') \in 
H^q(D_j, K_{D_j} \otimes \hat{D}_{ij} \otimes F )=0
$$
by using the vanishing theorem again, 
where $\hat{D}_{ij}:=D_{J} - (D_i+D_j)$. 
Hence we can take $\beta'' \in H^q(X, K_X \otimes \hat{D}_{ij} \otimes F )$ 
such that $\beta = \theta_{i}( \theta_{j}(\beta''))=\theta_{ij}(\beta'')$, 
where $\theta_{ij}$ is the morphism induced by the effective divisor $D_i+D_j$. 
By repeating this process, 
we can conclude that $\alpha=\theta_K(\beta)=\theta_K(\theta_J(\gamma))=\theta_D (\gamma)$ 
for some $\gamma \in H^q(X, K_X  \otimes F )$. 
This completes the proof by \cite[Theorem 1.6]{Mat17}. 
\end{proof}


\end{document}